\documentclass{article}

\usepackage{amsmath}
\usepackage{amssymb}
\usepackage{amsthm}
\usepackage[latin1]{inputenc}
\usepackage[english]{babel}
\usepackage[T1]{fontenc}

\usepackage{xcolor}
\usepackage{comment}

\usepackage{latexsym,color,graphicx,enumerate}
\usepackage{hyperref}

\newtheorem{thm}{Theorem}[section]
\newtheorem{prop}[thm]{Proposition}
\newtheorem{coro}[thm]{Corollary}
\newtheorem{lemma}[thm]{Lemma}
\newtheorem{rem}[thm]{Remark}

\newtheorem{conjecture}{Conjecture}[section]

\makeatletter

\newcommand*{\plim}[1][]{%
	\if\relax\detokenize{#1}\relax
	\def\next{\qopname\relax m{lim}}%
	\else
	\def\next{\qopname\newmcodes@ m{#1-lim}}%
	\fi
	\next
}

\newcommand*{\psum}[1][]{%
	\DOTSB
	\if\relax\detokenize{#1}\relax\else
	\operatorname{#1-}\mkern-\thinmuskip
	\fi
	\sum@\slimits@
}

\makeatother

\newcommand{\R}{\mathbb{R}}             
\newcommand{\N}{\mathbb{N}}             
\newcommand{\Z}{\mathbb{Z}}             
\newcommand{\C}{\mathbb{C}}             

\newcommand{\half}{\frac{1}{2}}

\newcommand {\clb}{\color{black}}

\numberwithin{equation}{section}

\title{On the magnetic Dirichlet to Neumann operator on the disk\\-- strong diamagnetism and strong magnetic field limit--}

\author{Bernard Helffer  and Fran{\c{c}}ois Nicoleau}

\begin{document}

\maketitle

\begin{abstract}
Inspired by a paper by T. Chakradhar, K. Gittins, G. Habib and N. Peyerimhoff, we analyze their conjecture that the ground state energy of the magnetic Dirichlet-to-Neumann operator on the disk tends to $+\infty$ as the magnetic field tends to $+\infty$. This is an important step towards the analysis of the curvature effect in the case of general domains in $\mathbb R^2$.
 \end{abstract}




\section{Introduction}

In this paper, we study the ground state energy of the Dirichlet to Neumann operator in the case with constant magnetic field and focus on the case of the    unit disk $D(0,1) \subset \R^2$.

\vspace{0.1cm}\noindent
 Let $\Omega$ be any bounded domain of $\R^n$, $n \geq 2$, with smooth boundary. For any $ u \in \mathcal D'(\Omega)$, the magnetic Schr\"odinger operator  on $\Omega$ is defined as 
\begin{equation}\label{defMagOp}
H_{A}\  u = (D-A)^2 u = -\Delta u -2i \   A \cdot \nabla u  + (A^2 -i \  {\rm{div}} \ A ) u,
\end{equation}
where   $D= -i \nabla$,  $-\Delta$ is the usual positive Laplace operator on $\R^n$ and ${\displaystyle{A= \sum_{j=1}^n A_j dx_j}}$ is the 1-form magnetic potential. We often identify the $1$-form magnetic potential $A$ with the vector field 
$\overrightarrow{A} = (A_1, ..., A_n)$.  

\vspace{0.1cm}\noindent
In what follows, we assume that $A \in C^{\infty}(\overline{\Omega}; \R^n)$. The magnetic field is given by the $2$-form $B =dA$.  

\vspace{0.1cm}\noindent
Since zero does not belong to the spectrum of the Dirichlet realization  of  $H_{A}$,  the boundary value  problem 
\begin{equation}  \label{Dirichlet}
	\left\{
	\begin{array}{rll}
		H_{A} \  u &=&0  \  \ \rm{in}  \ \ \Omega,\\
		u_{  \vert \partial \Omega} & =& f \in H^{1/2}(\partial\Omega) ,
	\end{array}\right.
\end{equation}
has a unique solution $u \in H^1(\Omega)$. The  Dirichlet to Neumann map, (in what follows D-to-N map), is defined  by
\begin{equation} \label{D-to-N--map}
\begin{array}{rll}
\Lambda_{A} :   H^{1/2}(\partial\Omega)& \longmapsto & H^{-1/2}(\partial\Omega) \\ 
	               f   &\longmapsto&  \left(\partial_{\nu} u + i \langle A, \vec{\nu} \rangle \ u  \right)_{  \vert \partial \Omega} ,
\end{array}
\end{equation}
where $\vec{\nu}$ is the outward normal unit vector field on $\partial\Omega$.  More precisely, we define the D-to-N map using the equivalent weak formulation : 
\begin{equation}\label{DtNweak}
\left\langle \Lambda_{A} f , g \right \rangle_{H^{-1/2}(\partial \Omega) \times H^{1/2}(\partial \Omega)} = \int_\Omega  \langle (D-A)u , (D-A)v \rangle\ dx\,,
\end{equation}
for any $g \in H^{1/2}(\partial \Omega)$ and $f \in H^{1/2}(\partial \Omega)$ such that $u$ is the unique solution of (\ref{Dirichlet}) and $v$ is any element of $H^1(\Omega)$ so that $v_{|\partial \Omega} = g$. Clearly, the D-to-N map is a positive operator.

\vspace{0.1cm}\noindent
We recall that the spectrum of the D-to-N operator is discrete and is given by an increasing sequence of eigenvalues 
\begin{equation} \label{spectrum}
 0 \leq 	\mu_1 \leq \mu_2 \leq ... \leq \mu_n \leq ...  \to + \infty\,. 
\end{equation}

\vspace{0.4cm}\noindent
In this paper, we consider the foregoing issue in the particular case of a constant magnetic field of strength $2b$ in the unit disk $D(0,1) \subset \R^2$, ($b \in \R$).  We denote by $\Lambda^{DN}(b)$ the associated D-to-N map and
by $\lambda^{DN} (b)$ its ground state energy.

\vspace{0.1cm}\noindent
We were inspired by the recent work of T. Chakradhar, K. Gittins, G. Habib and N. Peyerimhoff, (see \cite{CGHP}, Example 2.8).  If  $\lambda^{DN}(b)$ is the ground state energy  of the D-to-N map $\Lambda^{DN}(b)$, these authors have conjectured  that  $\lambda^{DN}(b) $ tends to $+\infty$ as the magnetic field $b$ tends to $+\infty$.

\vspace{0.2cm}\noindent
In what follows, we give a complete answer to their conjecture and our main result is the following :

\begin{thm}\label{main}
	One has the asymptotic expansion as $b \to + \infty$,
	\begin{equation}
		\lambda^{DN}(b) = \alpha b^{1/2} - \frac{\alpha^2 +2}{6} + \mathcal O (b^{-1/2})\,,
	\end{equation}
where $-\alpha$ is the unique negative zero  of the so-called parabolic cylinder function $D_{\half} (z)$ .
\end{thm}

\vspace{0.2cm}\noindent
We recall that the parabolic cylinder functions $D_\nu (z)$ are the (normalized)  solutions of the differential equation 
\begin{equation}\label{ODEDnu0}
	w'' + (\nu  + \frac 12 - \frac{z^2}{4})\ w=0\,,
\end{equation}
which tend to $0$ as $z \to +\infty$. We refer to Section 2 for more details on the parabolic cylinder functions. 

\vspace{0.1cm}\noindent
At last, the positive real $\alpha$ appearing in this theorem is approximately equal to 
\begin{equation}\label{approalpha}
	\alpha = 0.7649508673 ....
\end{equation}

{\clb 
\begin{rem}
	In absence of a magnetic field, it is well known that there exists a constant $C>0$ (depending on the geometry of the domain) such that $|\mu_k - \sqrt{\lambda_k}| \leq C$ for all $k \geq 1$, where the 
	$\lambda_k$  are the eigenvalues of the Dirichlet (or Neumann) Laplacian on the boundary $S^1$. Theorem  \ref{main} implies that this is not the case in presence of  a magnetic field $b$ since the eigenvalues $\lambda_k(b)$ are periodic in the variable $b$. We also refer to  (\cite{HKN}, Theorem 1.2) in the case of general domains.
\end{rem}
}
\vspace{0.2cm}\noindent
Our second result is concerned by diamagnetism. We recall that by diamagnetism we mean that $\lambda^{DN}(b)$ is minimal for $b=0$. This result has been proved in full generality in \cite{EO} (see also \cite{HeNi} for variants of this result and the last section). We now introduce the definition of strong magnetism by the property that $[0,+\infty)\ni b \mapsto \lambda^{DN}(b)$ is non decreasing.

\vspace{0.1cm}\noindent
We will actually prove  the stronger :

\begin{thm}
The map $b \mapsto \lambda^{DN}(b) $ is increasing on $(0,+\infty)$. 
\end{thm}

\vspace{0.3cm}
\noindent \textbf{Acknowledgements} :  We are grateful to Katie Gittins, Ayman Kachmar and the anonymous referees for their helpful comments on this work.\\

\newpage

\section{ Constant magnetic field in the disk} \label{champconstant}

\subsection{Framework.}

We consider the following magnetic $1$-form defined in the unit disk $D(0,1) \subset \R^2$ defined by :
\begin{equation}\label{constantb}
	A(x,y) = b (-y dx + xdy),
\end{equation}
where $b$ is a fixed real constant. The $2$-form $dA= 2b\  dxdy$ is a constant magnetic field of strength $2b$. Note that the magnetic potential $A$ satisfies the Coulomb gauge  ${\rm div} A=0$,  and $\langle A, \nu \rangle =0$ on the boundary of the unit disk  $\partial D (0,1)=S^{1}$. 
The magnetic Laplacian associated with  this potential $A$ is given by 
\begin{equation}\label{hamiltonian}
	H_A = (D-A)^2 .
\end{equation}
In order to construct the D-to-N map, first we have to solve :
\begin{equation}\label{eq:2.3}
\left\{
\begin{array}{ll}
	H_A \ v =0  & \mbox{in} \  D(0,1) , \\
	v= \psi  & \mbox{on} \ S^1 .
\end{array}
\right.
\end{equation}
Then, in polar coordinates $(r, \theta)$, the D-to-N map is defined (in a weak sense) by :
\begin{equation}\label{defD-to-N--}
\begin{array}{lll}
\Lambda^{\rm DN}(b) : H^{\half} (S^1) & \to &    H^{-\half} (S^1) \\
	 \hspace{1.5cm} \psi  &\to& \partial_r v (r, \theta)|_{r=1} \,,
\end{array}
\end{equation}
where $v$ is the solution of \eqref{eq:2.3} expressed in polar coordinates.\\

\noindent
With the help of the Fourier decomposition,  
\begin{equation}
v(r, \theta) = \sum_{n \in \Z} v_n (r) e^{in \theta}\ \ , \ \  \psi(\theta) = \sum_{n \in \Z} \psi_n e^{in \theta}, 
\end{equation}
we see, (\cite{CLPS1}, Appendix B), that $v_n (r)$ solves the following ODE :
\begin{equation}\label{polarequations}
	\left\{
	\begin{array}{ll}
		- v_n'' (r) - \frac{v_n'(r)}{r} + (br-\frac{n}{r} )^2 v_n (r)= 0   & \mbox{for} \  r \in (0,1) , \\
		v_n (1) = \psi_n .
	\end{array}
	\right.
\end{equation}
A bounded solution to the differential equation (\ref{polarequations}) is given by, (see \cite{CLPS1}, Eq. (B.2)):
\begin{equation}\label{solutionpos}
	v_n(r) = e^{-\frac{br^2}{2}} r^n L_{-\half}^n (br^2)   \ \ \mbox{for} \ n \geq 0 ,
\end{equation}
where $ L_{\nu}^\alpha (z)$  denotes the generalized Laguerre function. For $n \leq -1$, thanks to symmetries in  (\ref{polarequations}), we get a similar expression for $v_n(r)$  
changing the parameters $(n,b)$ into $(-n, -b)$.

\vspace{0.3cm}\noindent
We recall, (\cite{MOS1966}, p. 336), that the  generalized Laguerre functions $L_{\nu}^\alpha (z)$ satisfy the differential equation:
\begin{equation}
	z \frac{d^2 w}{dz^2} + (1+\alpha-z) \frac{dw}{dz} + \nu w =  0 ,
\end{equation}
and are given by
\begin{equation} \label{lienM}
	L_{\nu}^\alpha (z) = \frac{\Gamma(\alpha+ \nu+1)}{\Gamma (\alpha+1)\Gamma(\nu+1) }  M(-\nu,\alpha+1,z)\,.
\end{equation}
We recall also that $M(a,c,z)$ is the Kummer confluent hypergeometric function, (this function is also denoted by $_1F_1 (a,c,z)$ in the literature), and is defined as 
\begin{equation}\label{Kummerfunction}
	M(a,c,z) = \sum_{n=0}^{+\infty} \frac{(a)_n}{(c)_n} \ \frac{z^n}{n!}\,.
\end{equation}

\vspace{0.1cm}\noindent
Here $z$ is a complex variable, $a$ and $c$ are parameters which can take arbitrary real or complex values, except that $c \notin \Z^-$. At last, 
{\clb 
\begin{equation}
(a)_0 =1 \ ,\ (a)_n = \frac{\Gamma(a+n)}{\Gamma(a)} = a(a+1) ...(a+n-1), \ n=1,2, ... \,,
\end{equation}
}
are the so-called Pochhammer's symbols, (see \cite{MOS1966}, p. 262). Finally, the  function $M(a,c,z)$ satisfies the differential equation :
\begin{equation}\label{KummerODE}
	z\frac{d^2 w}{dz^2}  + (c-z)\frac{dw}{dz} -a w=0\,.
\end{equation}

\vspace{0.1cm}\noindent
For $0 < a < c$ and $n \in \N$,  we notice that, 
\begin{eqnarray}
	 \frac{(a)_n}{(c)_n} &=& \frac{\Gamma(c)}{\Gamma(a) \Gamma (c-a)} \ \beta(a+n, c-a)\, \nonumber \\
	                     &=& \frac{\Gamma(c)}{\Gamma(a) \Gamma (c-a)} \ \int_0^1 t^{a+n-1} \ (1-t)^{c-a-1}\ dt\,.
\end{eqnarray} 
Here, $\beta(x,y)$ is the well-known Beta function. It follows easily that, for $0 < a < c$, we have the  following integral representation (\cite{MOS1966}, p. 274) which we can take as the definition :
\begin{equation}\label{eq:kummer}
M (a,c,z)= \frac{\Gamma(c)}{\Gamma(c-a)\Gamma(a)} \int_0^1 e^{zt} t^{a-1} (1-t)^{c-a-1} dt\,.
\end{equation}
The derivative of the Kummer function 
with respect to $z$ is given by :
\begin{equation}\label{derivM}
 \partial_z M(a,c,z) := M'(a,c,z)=\frac{a}{c} \  M(a+1,c+1,z)\,.
\end{equation}

\par\noindent
For the convenience of the reader, we recall also  the following relations between two contiguous Kummer's functions, (see \cite{MOS1966}, p. 265-267) :

\begin{lemma}\label{contiguous}
For $a,z \in \C$ and $c \notin \Z$, one has :
\begin{eqnarray}
	&(i)&  (c-1) \ M(a, c-1,z) + (a+1-c)\ M(a,c,z) - a\ M(a+1,c,z)=0\,. \nonumber \\
	&(ii)& c\ M(a,c,z)-c\ M(a-1,c,z) -z\  M(a, c+1,z)=0\,. \nonumber \\
	&(iii)& a \ M(a+1,c,z) -a\ M(a,c,z)- z\ M' (a,c,z)=0 \,. \nonumber  \\
	&(iv)& (c-a)\ M(a-1,c,z) + (z+a-c) \ M(a,c,z) -z \ M' (a,c,z)=0\,. \nonumber
\end{eqnarray}
\end{lemma}

\begin{proof}
	Formulas $(i)-(ii)$ can be verified by direct substitution of the series (\ref{Kummerfunction}).  The other recursion relations $(iii)-(iv)$ can be obtained in the same way  using (\ref{derivM}).
\end{proof}

\vspace{0.1cm}\noindent
 Now, let us return to the study of the eigenvalues of the D-to-N map $\Lambda(b)$. They are usually called {\it magnetic Steklov eigenvalues}.  Obviously, they are given by
\begin{equation}
	\lambda_n = \frac {v_n' (1)}{v_n (1)} \ \ \mbox{for} \ n \in \Z \,.
\end{equation}
Thus, using (\ref{solutionpos}) and (\ref{lienM}),  we see that the {\it magnetic Steklov spectrum} is the set :
\begin{equation}\label{spectrumSpos}
	\sigma(\Lambda^{\rm DN}(b)) = \{\lambda_0(b)\} \cup \{\lambda_n(b) , \lambda_n(-b) \ \}_{n \in \N^*},
\end{equation}
{\clb $\N^*= \mathbb N\setminus \{0\}$ being the set of positive natural numbers, and where for $n \geq 0$,}
\begin{equation}\label{explicitvp}
\lambda_n (b) =  n - b +  2b \ \frac{M'(\half, n+1, b)}{M(\half, n+1,b)}.
\end{equation}
 Since the D-to-N map is a positive operator, one gets $\lambda_n (b) \geq 0$ for all $ n\in \Z, b \in \R$, and in what follows, we shall see that $\lambda_n(b)=0$ if and only if $n=0$ and $b=0$. 

\vspace{0.2cm}\noindent
In this paper, motivated by questions in (\cite{CGHP}, Example 2.8), we are interested in the analysis of
\begin{equation}
	\lambda^{\rm DN}(b) :=\inf_{n\in \Z} \lambda_n(b)\,,
\end{equation}
as $b\rightarrow +\infty$, (see Figure 1 which appears in \cite{CGHP}). We will prove later  as a consequence of the inequality
\begin{equation}\label{eq:n-n}
\lambda_n(b) \leq \lambda_{-n}(b) \,,\,  \forall b \geq 0, \forall n\geq 0\,,
\end{equation}
that actually
\begin{equation}\label{eq:infngeq0}
	\lambda^{DN}(b)  :=\inf_{n\in \N} \lambda_n(b)\,.
\end{equation}

\begin{figure}
	\begin{center}
		\includegraphics[width=0.48\textwidth]{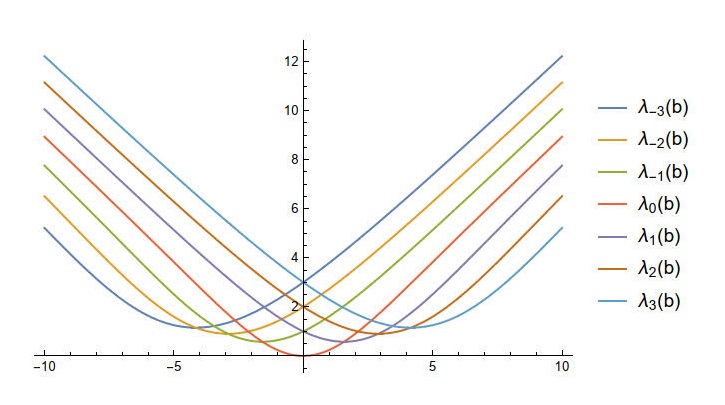}
		\includegraphics[width=0.48\textwidth]{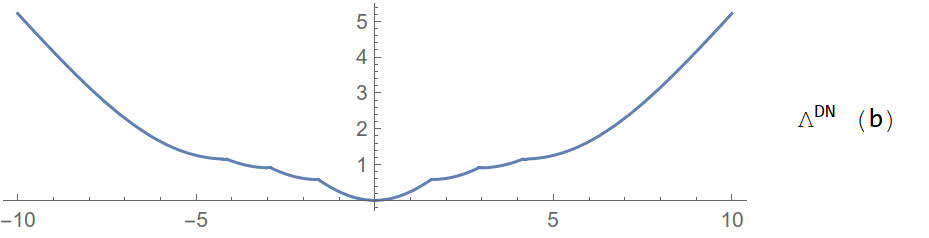} 
	\end{center}
\caption{The magnetic Steklov eigenvalues $\lambda_n (b)$ (left) and  the ground state energy $\lambda^{DN}(b) $ (right). }
\end{figure}

\newpage
\subsection{Analogies with the magnetic Neumann realization on the disk.}
{\clb The analysis as $b\rightarrow 0$ (the weak field limit) is easy since  $\lambda^{\rm DN}(b)=\lambda_0 (b)$, and we can then use  a regular perturbation argument as it was done in (\cite{HKN}, Section 1.5) in the case of arbitray regular domains.}

\vspace{0.1cm}\noindent
The analysis for $b$ large will be parallel with what has been done in \cite{HeLe} where the authors analyze the large $b$ behavior of the  ground state energy of the Neumann realization of the magnetic Laplacian in the disk.

\vspace{0.1cm}\noindent
Roughly speaking our goal in this paper is to analyze the conjecture  that, as $b \rightarrow + \infty $,
 \begin{equation}
 \lambda^{DN}(b)  \sim \alpha \sqrt{b} \,,
 \end{equation}
where $\alpha$ is a suitable positive constant.

\vspace{0.1cm}\noindent
The hope is to get this kind of estimates by analyzing carefully the intersection points between the graph of $\lambda_n (b)$ and the graph of $\lambda_{n+1}(b)$. We will show its uniqueness, denote it by $z_n$, and will analyze the asymptotics of $z_n$ and $\lambda_n (z_n)$ as $n\rightarrow +\infty$.

\vspace{0.1cm}\noindent
This problem has a lot of analogies with the analysis of the Neumann realization of  $H_A$ as achieved in \cite{HeLe}, (see also there  the long list of references  therein 
since the physicist De Gennes {\clb  \cite{BH,DH1,DG,FH4,HeMo,S-J}}   ), whose ground state $\lambda^{DG}(b)$ has the behavior  as $ b \to \infty$, (see Figure 3),
\begin{equation}\label{eq:2.22}
  \lambda^{DG}(b)\sim 2 \Theta_0 b\,.
\end{equation}
Here, $\Theta_0$ is the so-called De Gennes constant which is equal  (see \cite{MPS} and below) approximately  to 
  \begin{equation}\label{eq:2.23}
  \Theta_0\approx 0.5901061249 ....
  \end{equation}
More precisely, one has :
\begin{equation}\label{eq:2.24}
 \Theta_0 =\xi_0^2 \,.
\end{equation}
We shall see in Remark \ref{eqPersson} that $\xi_0 \approx 0.76818$ is the unique solution in the interval $[0,1]$ of the equation 
\begin{equation}\label{eqDeGennes}
	f(\xi) := \xi D_{(\xi^2-1)/2} (-\sqrt{2} \xi) + \sqrt{2} D_{(\xi^2 +1)/2} (-\sqrt{2} \xi) =0 \,,
\end{equation}
(see Figure 2), and this equation was used by M. Persson Sundquist in (\cite{MPS}) to get over 200 decimals of $\Theta_0$.
\begin{figure}
	\begin{center}
		\includegraphics[width=0.45\textwidth]{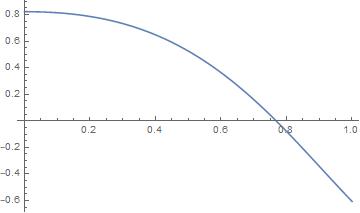}
	\end{center}
	\caption{Graph of $f(\xi)$. }
\end{figure}

\vspace{0.1cm}\noindent
In (\ref{eqDeGennes}), $D_\nu (t)$ is  the  parabolic cylinder function satisfying the differential equation, (see \cite{MOS1966}, p. 324) :
\begin{equation}\label{ODEDnu}
w'' + (\nu  + \frac 12 - \frac 14 t^2)w=0\,.
\end{equation}
We recall that, for any $\nu <0$, one has the following integral representation (\cite{MOS1966}, p. 328) :
\begin{equation}\label{integralrep}
	D_\nu (z) = \frac{e^{- \frac{z^2}{4}}}{\Gamma(-\nu)} \ \int_0^{+\infty} t^{-\nu -1} e^{- ( \frac{t^2}{2} +zt)} \ dt \,.
\end{equation} 
The parabolic cylinder functions $D_\nu (z)$ satisfy the recurrence relations  (\cite{MOS1966}, p. 327),
\begin{subequations}
\begin{equation}\label{recurrenceDnu}
 D'_\nu (z) - \frac{z}{2} D_\nu (z) + D_{\nu+1}(z)=0 \,,
\end{equation}
\begin{equation}\label{recurrenceDnu1}
	D_{\nu+1} (z) - z D_\nu (z) + \nu D_{\nu-1}(z)=0 \,.
\end{equation}
\begin{equation}\label{recurrenceDnu2}
	D'_{\nu} (z) + \frac{z}{2} D_\nu (z) -\nu  D_{\nu-1}(z)=0 \,,
\end{equation}
\end{subequations}
\vspace{0.1cm}\noindent
At last, the parabolic cylinder functions have the following asymptotic expansion, (\cite{MOS1966}, p. 331) :
\begin{equation}\label{asymptDnu}
	D_\nu (z) = e^{- \frac{z^2}{4} } z^\nu \ (1+ \mathcal O (\frac{1}{z^2})) \ ,\ z \to + \infty.
\end{equation}
Notice that for $\nu <0$, these asymptotics are obtained by applying the Laplace integral method in \eqref{integralrep}.

\begin{figure}
\begin{center}
	\includegraphics[width=0.7\textwidth]{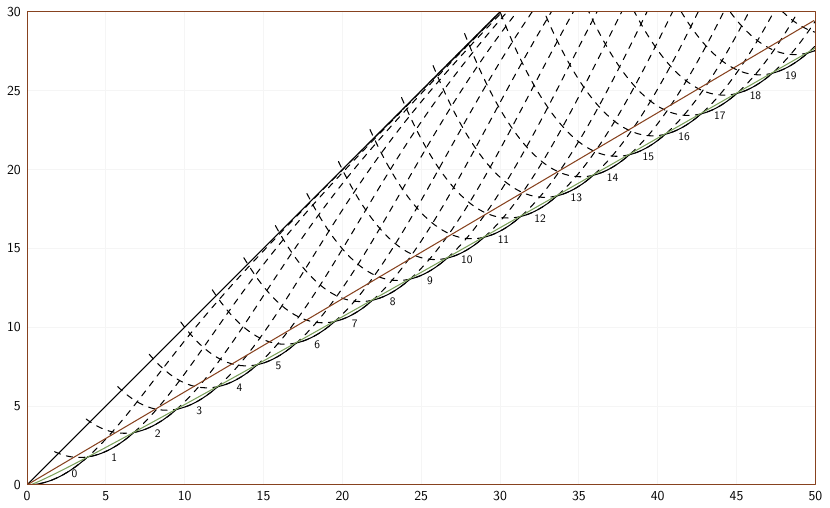}\label{StJ}
	\caption{Graph proposed by Saint-James.}
\end{center}
\end{figure}

\vspace{0.3cm}\noindent
Now, in order to explain how we get (\ref{eqDeGennes}), let us start from the standard harmonic oscillator $D_t^2 + t^2$, and let us consider 
\begin{equation}
\check D_\nu(t):= D_\nu (\sqrt{2} t)\,.
\end{equation}
We first observe, using \eqref{ODEDnu}, that 
$\check D_\nu(t)$ is a solution (in the distributional sense) in $\mathbb R$ of the ODE
\begin{equation}
(D_t^2 +t^2 -\mu) u=0\,.
\end{equation}
Clearly, using (\ref{asymptDnu}), we see that $\check D_\nu(t)$ tends to $0$ as $t \to +\infty$, and one has $\mu=  2 \nu + 1$.

\vspace{0.2cm}\noindent
To analyze the spectral problem for  the Neumann problem for the De Gennes operator
$$
\mathfrak h(\xi) = D_t^2 + (t-\xi)^2
$$
with {\it{a special functions approach}}, we consider the map :
\begin{equation}
t \mapsto u_\xi (t):=\check D_\nu (t-\xi)\,,
\end{equation}
and we impose the Neumann condition at $t=0$ :
\begin{equation}\label{eq:2.32}
u_\xi'(0)=\check D_\nu' (-\xi) =0\,.
\end{equation}
This gives an implicit relation between $\nu$ and $\xi$, and with  any solution $\nu$ of this equation, we can associate an eigenvalue of $\mathfrak h (\xi)$. We choose the lowest one, denoted  $\lambda^{DG}(\xi)$, and we denote
\begin{equation}\label{2.32}
\lambda^{DG}(\xi)= 2 \nu(\xi) + 1\,.
\end{equation}
We recall from \cite{HeMo} that $\xi \mapsto \lambda^{DG} (\xi)$ has unique minimum at $\xi=\xi_0$, and we have at this minimum
\begin{equation}
\lambda^{DG} (\xi_0)= \Theta_0\,,
\end{equation}
where $\Theta_0$ already appears in \eqref{eq:2.22}-\eqref{eq:2.24}, and one has :
\begin{equation}
(\lambda^{DG})' (\xi_0)=0\,.
\end{equation}
So using \eqref{eq:2.32}-(\ref{2.32}), we get
\begin{equation}\label{stat}
\nu'(\xi_0)=0\,,
\end{equation}
and we have
\begin{equation}
\Theta_0 = 2 \nu(\xi_0) +1\,,
\end{equation}
which implies

\begin{equation}\label{eq:nunu}
	\nu (\xi_0) = \frac{\xi_0^2 -1}{2}\,,
\end{equation} 
and coming back to the functions $D_\nu(t)$, we obtain from (\ref{eq:2.32}) :
\begin{equation}\label{cond1}
D'_{\nu(\xi_0)}(- \sqrt{2}\ \xi_0)=0\,.
\end{equation}

\begin{rem}\label{eqPersson}
Using \eqref{recurrenceDnu}, the condition (\ref{cond1})	can be also reformulated as 
	\begin{equation}
	f(\xi_0)	 =\xi_0 D_{(\xi_0^2-1)/2} (-\sqrt{2} \xi_0) + \sqrt{2} D_{(\xi_0^2 +1)/2} (-\sqrt{2} \xi_0) =0 \,,
	\end{equation}
as mentioned previously in (\ref{eqDeGennes}).
\end{rem}

\vspace{0.1cm}\noindent


\vspace{0.3cm}\noindent
Finally, we remark that the first normalized eigenfunction for $\mathfrak h_0(\xi_0)$ reads
\begin{equation}\label{eigenf}
u_0 (t)= \check D_{\nu(\xi_0)} (t-\xi_0)/ || \check D_{\nu(\xi_0)} (\cdot-\xi_0)||_{L^2(\mathbb R^+)} \,.
\end{equation}
For this eigenfunction, we get
\begin{equation} \label{eigenf0}
u_0(0)=  \check D_{\nu(\xi_0)} (-\xi_0) / || \check D_{\nu(\xi_0)} (\cdot-\xi_0)||_{L^2(\mathbb R^+)}\,.
\end{equation}

\section{The case of the half-plane.}

In parallel with the analysis of the Neumann problem for an open set, it is natural to first understand the case of the half-space. Hence we analyze the Dirichlet problem for the operator $ D_t^2 + (D_x-2bt)^2$ 
in $\mathbb R^2_+:=\{(t,x) \in \mathbb R^2,\ t>0\}$ and the corresponding D-to-N operator.
  
\vspace{0.1cm}\noindent  
We recall that for the Neumann problem with $b=1$, the bottom of the spectrum is given by $2 \Theta_0$.
 
\vspace{0.2cm}\noindent
Now, let us compute for $b>0$  the associated D-to-N map which is denoted $\Lambda^{\rm DN} (b)$ or shortly by $\Lambda(b)$. First, we have to solve
\begin{equation}  \label{DirichletHalfspace}
	\left\{
	\begin{array}{rll}
		( D_t^2 + (D_x-2bt)^2) \Psi_b(t,x)&=&0 ,\\
		\Psi_b(0,x)& =& \phi_0(x)\,.
	\end{array}\right.
\end{equation}
Then, we define the D-to-N map $\Lambda(b)$ as
\begin{equation}
    \Lambda (b) : \phi_0 \mapsto - \partial_t \Psi_b (0,\cdot )\,.
\end{equation}
{\clb Taking the partial Fourier transform of $\phi_0$ in the $x$ variable, (i.e $\hat \phi_0 (\xi)  = \frac{1}{\sqrt{2\pi}} \int_\R e^{-ix\xi} \phi_0(x) \ dx$)}, we have to study the map :
\begin{equation}
  \widehat \Lambda(b) :  \hat \phi_0 \mapsto - \partial_t \hat \Psi_b (0,\cdot)\,,
\end{equation}
where  $\hat \Psi_b$ is the solution of
\begin{equation}
  ( D_t^2 + (\xi-2bt)^2) \hat \Psi_b(t,\xi)=0\,,\, \hat \Psi_b(0,\xi)=\hat \phi_0 (\xi)\,.
\end{equation}
Finally, one easily gets :
\begin{equation} 
\widehat \Lambda(b) \hat \phi_0 (\xi) = f_b(\xi) \ \hat \phi_0(\xi)\,,
\end{equation}
where $f_b (\xi)$  is a Fourier multiplier which can be calculated explicitly, (see later for the details). \\

\vspace{0.2cm}\noindent
It follows that the D-to-N map $\Lambda (b)$ is a simple convolution operator and the spectrum of $\Lambda(b)$  is the closure of the image of $f_b$ in $\mathbb R$ :
\begin{equation}
   { \rm spect}(\Lambda(b)) =\overline{{\rm Ran} \,f_b}\,.
\end{equation}
Since we are interested in the bottom of the spectrum, we introduce
\begin{equation}
    m(b):=   \inf   {\rm spect}(\Lambda(b)) \,.
\end{equation}
To determine the Fourier multiplier $f_b(\xi)$ we have to solve :
\begin{equation}
  ( D_t^2 + (\xi-2bt)^2) \ \Phi_b(t,\xi)=0\,,\, \Phi_b(0,\xi)=1 \,.
\end{equation}
First, we observe that 
\begin{equation}
   \Phi_b(t,\xi) = \Phi_1(b^\frac 12 t, b^{-\frac 12} \xi)\,.
\end{equation}
This leads to
\begin{equation}
   f_b(\xi)= b^{\frac 12} f_1 (b^{-1/2} \xi)\,.
\end{equation}
Hence we get :
\begin{equation}
   m(b)=b^{1/2} \ m(1)\,.
\end{equation}
It remains to complete the computation for $b=1$.
\vspace{0.2cm}\noindent
In order to solve
\begin{equation}
  ( D_t^2 + (\xi-2t)^2)\ \Phi(t,\xi)=0\,,\, \Phi(0,\xi)=1 \,,
\end{equation}
{\clb we first remark that  $u(t) := D_{-1/2} (2 t)$ is a solution of  $-u''(t) + 4 t^2 u(t)=0$ which tends to $0$ at $+\infty$.}  Notice that this function can be also written as
\begin{equation}
	D_{-\frac 12} (z) = \sqrt{\frac{z}{2\pi}}\ K_{\frac 14} (\frac{z^2}{4})\,,
\end{equation}
where $K_{\nu}(z)$ is the  usual modified Bessel function of the third kind, (see (\cite{MOS1966}, p. 326).



\vspace{0.2cm}\noindent
Thus, we obtain
\begin{equation}
  \Phi(t,\xi) = \frac{D_{-\half} (2t-\xi)}{ D_{-\half} (-\xi)}\,,
\end{equation}
and we get finally
\begin{equation}
  f_1(\xi)= -2 \ \frac{D'_{-\half} (-\xi)}{ D_{-\half}(-\xi)}\,.
\end{equation}

\vspace{0.2cm}\noindent
We shall see that the miminum $m(1)$ is closely related to the unique negative zero of the parabolic cylinder function $D_{\half} (z)$, (see Figure 4). We recall that this function can be also expressed as, (see \cite{MOS1966}, p. 326),
\begin{equation}
D_\half (z)=\frac{1}{\sqrt{\pi}}\ \Big(\frac z2\Big)^{\frac 32} \left( K_{\frac 14}(\frac{z^2}{4}) + K_{\frac 34}(\frac{z^2}{4})\right)\,,
\end{equation}
and $D_{\half} (z)$ has a unique (negative) simple real zero denoted $-\alpha$. Using Mathematica, we  get \eqref{approalpha}.
	\begin{figure}
	\begin{center}
		\includegraphics[width=0.4\textwidth]{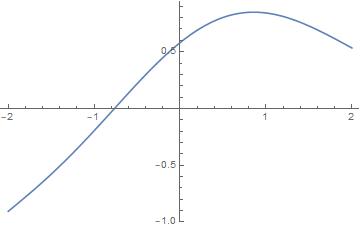}
	\end{center}
	\caption{Graph of the function $D_\half (x)$. }
\end{figure}

\vspace{0.1cm}\noindent
Now, we are able to state the following result :

\begin{prop}
The function $f_1(\xi)$ has a unique minimum which is attained at  $\xi = \alpha$ and one has $m(1) =\alpha$.
\end{prop}
    
\begin{proof}
The graph of $f_1$ suggests that $f_1$ has a unique minimum $x_0 \in [0.6\,,\, 0.8]$, (see Figure 5). Clearly, $f_1'(\xi)=0$ if and only if  $-D''_{-\half} (-\xi) D_{-\half} (-\xi) + (D'_{-\half} (-\xi))^2=0$. As the parabolic cylinder function $D_{-\half} (-\xi)$ satisfies the differential equation (\ref{ODEDnu}):
\begin{equation}
D''_{-\half} (-\xi) - \frac{1}{4} \xi^2 \ D_{-\half} (-\xi) =0\,,
\end{equation}
we see that $f'(\xi)= 0$ if and only if 
\begin{equation}
	\frac{1}{4} \xi^2\ (D_{-\half} (-\xi))^2 =  (D'_{-\half} (-\xi))^2\,.
\end{equation}
Numerically, we see that for $\xi \in [0.6\,,\, 0.8]$,  $D_{-\half} (-\xi) >0$ and  $D'_{-\half} (-\xi) <0$. It follows that
\begin{equation}\label{ptfixe}
	\half \xi  \  D_{-\half} (-\xi) = -  D'_{-\half} (-\xi)\,.
\end{equation}
So, using (\ref{recurrenceDnu}), we get immediateley
\begin{equation}
	 D_{\half} (-\xi) = 0\,,
\end{equation}
or equivalently $\xi = \alpha$ by definition. Using again (\ref{ptfixe}) with $\xi = \alpha$, we get $f_1 (\alpha) = \alpha$.

	\begin{figure}
		\begin{center}
			\includegraphics[width=0.4\textwidth]{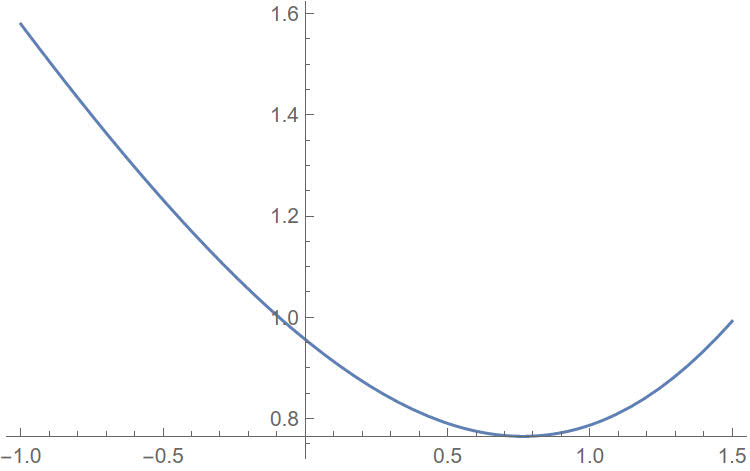}
		\end{center}
		\caption{Graph of the function $f_1(\xi)$. }
	\end{figure}
\end{proof}


\section{Intersecting points.}

As in \cite{HeLe}, for explaining the curves given in Figure 1, we will analyze the intersection points between the curves of the magnetic Steklov eigenvalues $\lambda_n (b)$ and $\lambda_{n+1}(b)$.
In this computation,  in view of \eqref{eq:n-n}, we restrict  our attention to the case of positive intersection points and $b$ is replaced by the variable $z$.

\subsection{Characterization of the intersection points.}
\vspace{0.2cm}\noindent
Let $z_n$ be the positive intersection point between the curves $\lambda_n (b)$ and $\lambda_{n+1}(b)$, {\clb (note that the existence of $z_n$ is actually proved in Corollary \ref{existencezn} ) } In other words, one has :  
 \begin{equation} \label{intersection}
   \lambda_n(z_n)=\lambda_{n+1} (z_{n})\,.
\end{equation}

\vspace{0.1cm}\noindent
Using \eqref{explicitvp} we obtain immediately :
\begin{equation}\label{eq:fn1}
 \frac{M(\frac 12,n+2,z)+ 2 z M'(\frac 12,n+2,z)}{M(\frac 12,n+2,z)} = 2z \frac{M'(\frac 12,n+1,z)}{M(\frac 12,n+1,z)}\,.
 \end{equation}
First, let us study the left hand side (LHS) of \eqref{eq:fn1}. Using Lemma \ref{contiguous} $(iii)$ with $a = \half$ and $c=n+2$, we get :
\begin{equation}
 M(\frac 32,n+2,z)=M(\frac 12,n+2,z)+ 2 z M'(\frac 12,n+2,z)\,.
\end{equation}
Hence we have :
 \begin{equation}
 (LHS)= \frac{ M(\frac 32,n+2,z)}{M(\frac 12,n+2,z)}\,.
\end{equation}
For the right hand side, using \eqref{derivM}), one has 
 \begin{equation}
 M'(\frac 12,n+1,z)=\frac{1}{2(n+1)} \ M(\frac 32,n+2,z)\,,
\end{equation}
 and this leads to
\begin{equation}
 (RHS)=\frac{z}{n+1} \;\frac{M(\frac 32,n+2,z)}{M(\frac 12,n+1,z)}\,.
\end{equation}
We consequently obtain at the intersection point :
{\clb 
\begin{equation}
 M(\frac 12,n+1,z)=\frac{z}{n+1} \ M(\frac 12,n+2,z)\,.
\end{equation}
}
Now, using Lemma \ref{contiguous} $(ii)$ with $a=\frac 12$, {\clb $c=n+1$}, we get :
{\clb 
\begin{equation}
 (n+1)M(\frac 12, n+1,z) -(n+1) M (-\frac 12,n+1,z) - z M(\frac 12,n+2,z)=0\,.
\end{equation}
}
Thus, we get $ \lambda_n(z)=\lambda_{n+1}(z)$ if and only if $M(-\frac 12,n+1,z)=0$.

\vspace{0.4cm}\noindent
Hence we have the following result :

\begin{prop}\label{characteriztion}
For any $n\geq 0$, there is a unique positive intersection point $z_n$  between the curves $\lambda_n (b)$ and $\lambda_{n+1}(b)$. Moreover, one has :
\begin{equation}
 M(-\frac 12,n+1,z_n)=0\
\end{equation}	
\end{prop} 
	
\begin{proof}
Using (\cite{DLMF}, 13.9.1), we see that the function $z \mapsto M(-\half; n+1, z)$ has a unique positive zero. This concludes the proof.
	\end{proof}

\begin{rem}
We emphasize that we can not use  the integral representation \eqref{eq:kummer} for studying the equation $M(-\frac 12,n+1,z_n)=0$. 
Nevertheless, there exists in the literature other integral representations (with  complex  contours), (see for instance  \cite{Te2}, (Eq. (22.5.69)) where the author uses a Hankel contour).
\end{rem}	

 \subsection{The (F) formula.}

 Using Lemma \ref{contiguous} $(iv)$ with $a = \half$ and $c = n+1$, we get immediately :
\begin{equation}
 (n+\frac 12) \ M(-\frac 12,n+1,z) + (z -n - \half) \ M (\frac 12,n+1,z) - z \ M'(\frac 12,n+1,z) =0\,.
\end{equation}
Hence Proposition \ref{characteriztion} implies for $z =z_n$, the following characterization :
\begin{equation}\label{eq:char1}
  (z -n - \half) \ M (\frac 12,n+1,z) - z \ M'(\frac 12,n+1,z) =0\,.
\end{equation}
Notice that if conversely $z$ is a solution of \eqref{eq:char1}, then it satisfies  $M(-\frac 12,n+1,z)=0$ and $z=z_n$.

\vspace{0.3cm}\noindent
{\clb Coming back to the definition of  $\lambda_n$ given in (\ref{explicitvp})}, we obtain the following formula :

\vspace{0.2cm}\noindent
\begin{lemma} \label{magic}
One has :
 \begin{equation}\label{eq:magic}
(F)\quad  \lambda_n (z_n) = z_n-n-1\,.
 \end{equation}
 \end{lemma}

\vspace{0.3cm}\noindent
This formula could play the role of the Saint-James formula in \cite{HeLe} but is actually much simpler.

\begin{rem}
Using $(F)$ formula,  we see  that for any $ n \geq 0$, one has :
	\begin{equation}
	z_n > n+1.
	\end{equation}
\end{rem}

\subsection{Computation of $\lambda'_n(z)$ and application to strong diamagnetism.}

We recall that 
\begin{equation}
	\lambda_n (z) = n-z +2z \ \frac{M'(\half, n+1, z)}{M(\half, n+1, z)}\,.
\end{equation}

\vspace{0.2cm}\noindent
We have the following elementary result :
\begin{prop}
We have 
\begin{eqnarray}\label{eq:lambdaprime}
\lambda'_n(z) &=&   \frac{M'(\half, n+1, z)}{(M(\half, n+1, z))^2} \ \left( M(\half, n+1, z) - (2n+1) M(-\half, n+1, z) \right) \label{lambdaprime1} \\
          &=& -2n\ \frac{M'(\half, n+1, z) \ M(-\half, n, z) }{(M(\half, n+1, z))^2}     \,. \label{lambdaprime2}
\end{eqnarray}
\end{prop}

\begin{proof}
In what follows, to simplify the notation, we set $M=M(\half, n+1, z)$. Clearly, one has :
\begin{equation}
\lambda'_n(z) = -1 + \frac{2M'}{M} + 2 \ \frac{zM''M -z (M')^2}{M^2}  \,.
\end{equation}	
As the confluent hypergeometric function $M(a,c,z)$ satisfies the differential equation (\ref{KummerODE}), a straightforward calculation gives :	
\begin{equation}\label{Eq1}
\lambda'_n(z) = \frac{2M'}{M^2} \left( (z-n)M  - zM' \right) \,.
\end{equation}
Using Lemma \ref{contiguous} $(iv)$, with $a= \half, \ c=n+1$, we get :
\begin{equation} \label{Eq2}
	(n+\half)\  M(-\half, n+1, z) = (n+ \half -z)\ M + z\ M' \,.
\end{equation}	
Plugging (\ref{Eq2}) into (\ref{Eq1}), we obtain easily (\ref{lambdaprime1}).	Now, using (\ref{lambdaprime1}) and  Lemma \ref{contiguous} $(i)$, we get 
(\ref{lambdaprime2}).
\end{proof}


\vspace{0.2cm}\noindent
As a by-product, we obtain :

\begin{coro}\label{existencezn}
For $n \geq 1$, one has $\lambda'_n (z_{n-1})	=0$. Moreover, $\lambda'_n (z)>0$ on $(z_{n-1}, + \infty)$ and $\lambda'_n (z)<0$ on $(0, z_{n-1})$ and $z_{n-1}$ is the unique minimum of $\lambda_n(z)$.
\end{coro}

\begin{proof}
The first equality is obvious using (\ref{lambdaprime2}) and Proposition \ref{characteriztion}.
On the other hand, since the application $z \mapsto M(-\half, n, z)$ is  a decreasing function, we see that, on the interval $(z_{n-1},+\infty)$, $M(-\frac 12,n,z)$ is negative whereas $ M'(\half, n+1, z)/(M(\half, n+1, z))^2$ is positive, {\clb (see \cite{DLMF}, (13.9.2))}. Hence $\lambda'_n(z)$ is positive on this interval. Looking now at the interval $(0,z_{n-1})$, we can prove in the same way that $\lambda'_n(z)$ is negative.
\end{proof}

\vspace{0.1cm}\noindent
Then we deduce : 
\begin{coro} For any $n \geq 1$, one has $z_{n-1} < z_{n}$ and 
 $\lambda_{n}(z)$ is increasing between $z_{n-1}$ and $z_{n}$.
\end{coro}

\vspace{0.3cm}\noindent
Finally we have the following result :

\begin{prop}
On the interval  $[z_{n-1},z_n]$, one has : 
\begin{equation}
\lambda^{DN}(z)= \lambda_n(z)\,.
\end{equation}
\end{prop}

\vspace{0.1cm}\noindent
As a final result, we obtain : 
 		
\begin{thm}
The map $z \mapsto \lambda^{DN}(z)$ is increasing on $(0,+\infty)$.
\end{thm}

\vspace{0.2cm}\noindent
Hence we have strong diamagnetism. Note that diamagnetism was proved for general domains in \cite{EO}. \\

\begin{prop}
	For $n \geq 1$, one has $\lambda_n'' (z_{n-1}) > 0$. 
\end{prop}

\begin{proof}
 Using Proposition \ref{characteriztion}, one has $M(-\half, n, z_{n-1})=0$. Thus, using (\ref{lambdaprime2}), a straightforward calculation shows that : 
\begin{equation}
\lambda_n'' (z_{n-1}) = -2n \ \frac{M'(\half, n+1, z_{n-1}) \ M'(-\half, n, z_{n-1}) }{(M(\half, n+1, z_{n-1}))^2}	
\end{equation}
Now, using (\ref{derivM}), we get : 
\begin{equation}
	M'(-\half, n, z_{n-1}) = - \frac{1}{2n} \ M(\half, n+1, z_{n-1})\,.
\end{equation}
It follows that :
\begin{equation}
\lambda_n'' (z_{n-1}) =  \frac{M'(\half, n+1, z_{n-1})}{M(\half, n+1, z_{n-1})}\,.
\end{equation}
So, using (\ref{explicitvp}), we get 
\begin{eqnarray}
\lambda_n'' (z_{n-1}) &=& \frac{\lambda_n (z_{n-1}) -n + z_{n-1}}{2 z_{n-1}}  \nonumber \\
	&=&  \frac{\lambda_{n-1} (z_{n-1}) -n + z_{n-1}}{2 z_{n-1}} \nonumber \\ &= &\frac{z_{n-1}-n}{z_{n-1}} >0\,,
\end{eqnarray}
where we have used the characterization of the intersection point (\ref{intersection})	and Lemma \ref{magic}.
\end{proof}


 \section{Asymptotics of $z_n$. }
 
 \subsection{Main result.} 
 
The main goal of this section is to prove the following asymptotic expansion :
\begin{prop}\label{completeasympt}
As $n\rightarrow +\infty$, $z_n$ admits the asymptotics
\begin{equation}\label{eq:twoterms}
z_n \sim n + \alpha  \sqrt{n} + \frac{\alpha^2 +2}{3} + \sum_{j\geq 1} \alpha_j n^{-\frac j2}\,.
\end{equation}
where $\alpha_j$ are suitable real constants.
\end{prop}

\subsection{An equivalent problem.}
Let us start from the characterization \eqref{eq:char1}, which is written in the form
\begin{equation}\label{eq:as1}
z-n-\frac 12= z\ \frac{M'(\frac 12,n+1,z) }{M(\frac 12,n+1,z)}\,,
\end{equation}
and we know that $z_n$ is a solution of this equation and is unique. We introduce as new variable :
\begin{equation} \label{eq:as2}
\beta = \frac{z-n-\frac 12}{\sqrt{n}}
\end{equation}
and we get that \eqref{eq:char1} is equivalent to :
\begin{equation}\label{eq:as3}
\Psi_n(\beta):=\beta\,  (1-\frac{1}{\sqrt{n} } \theta_n(\beta)) - (1 + \frac{1}{2n})\;  \theta_n(\beta)=0\,,
\end{equation}
where
{\clb 
\begin{equation}\label{eq:as4}
\theta_n(\beta) = \sqrt{n} \;  \ \frac{M'(\frac 12,n+1,z(\beta)) }{M(\frac 12,n+1,z(\beta))}\,,
\end{equation}
and  $z(\beta)= n+\frac{1}{2} + \sqrt{n} \beta$.\\
}
Hence $\beta_n:= \frac{z_n -n-\frac 12}{\sqrt{n}}$ is the unique solution of \eqref{eq:as3} and Proposition \ref{completeasympt} will be a consequence of  the existence of a sequence $\hat \alpha_j $ such that, as $n \rightarrow +\infty$, 
\begin{equation}\label{eq:twotermsz}
\beta_n \sim  \alpha   + \frac{2\alpha^2 +1}{6} n^{-1/2} + \sum_{j\geq 2} \hat \alpha_j n^{- j/2}\,,
\end{equation}
where the $\hat \alpha_j$ are suitable real constants.\\

\subsection{Analysis of $\theta_n$.}
Coming back to the formulas for $M(\frac 12,n+1,z)$ and $M'(\frac 12,n+1,z)$, (see \eqref{eq:kummer}-\eqref{derivM}), and after a change of variable $s =\sqrt{n}\, t$ in the defining integrals, we obtain :
\begin{equation}
\theta_n(\beta) = \frac{\sigma_n(\beta)}{\tau_n(\beta)}\,,
\end{equation}
where
\begin{eqnarray}
\sigma_n(\beta) &=& \int_0^{\sqrt{n}} e^{\big(\beta + n^{1/2}  + \frac 12 n^{-1/2} \big)s}\,s^{1/2} (1-s\,n^{-1/2})^{n-\frac 12} \ ds\,,\\
\tau_n(\beta) &=& \int_0^{\sqrt{n}} e^{\big(\beta + n^{1/2}  + \frac 12 n^{-1/2} \big)s}\,s^{-1/2} (1-s\,n^{-1/2})^{n-\frac 12} \ ds\,.
\end{eqnarray}
We now treat the asymptotics for $\sigma_n$ and $\tau_n$ separately but focus on $\sigma_n$ since the proof for $\tau_n$ is identical.

\vspace{0.1cm}\noindent
We first decompose the integral from $0$ to $n^{\frac 14}$ and then from $n^{1/4}$ to $n^{1/2}$. A rough estimate shows that
\begin{equation}
\sigma_n(\beta) = \int_0^{n^{1/4}} e^{\big(\beta + n^{1/2}  + \frac 12 n^{-1/2} \big)s}\,s^{1/2} (1-s\,n^{-1/2})^{n-\frac 12} ds + \mathcal O (n^{-\infty})\,.
\end{equation}
Here we have used that

\begin{equation}
\big(\beta + n^{1/2}  + \frac 12 n^{-1/2} \big)s +(n-\frac 12)  \log (1-s\,n^{-1/2})\leq \beta s + n^{-1/2} s -\frac 12 s^2+ \frac{1} {4n}  s^2\,,
\end{equation}
which permits to show that,  for $s\geq C n^{1/4}$, ($C>0$ is a suitable constant), and $n$ large enough, the right hand side inside the integral  is less than $e^{-\sqrt{n}/8} e^{-s^2/8}$.

\vspace{0.1cm}\noindent
We now consider
\begin{equation}
\widetilde \sigma_n(\beta): = \int_0^{C n^{1/4}} e^{\big(\beta + n^{1/2}  + \frac 12 n^{-1/2} \big)s}\,s^{1/2} (1-s\,n^{-1/2})^{n-\frac 12} ds\,.
\end{equation}
Using a Taylor expansion with remainder  of $\log (1-s\,n^{-1/2})$ to a sufficiently high order, we get  an infinite sequence of polynomials $P_j$ ($j\in \N^*$) such that for any {\clb $N \in \N^*$, there exists $p(N) \in \N^*$} such that 
\begin{equation}
\widetilde \sigma_n(\beta) = \int_0^{C n^{1/4}} e^{\beta s -\frac{s^2}{2}} s^{1/2} \Big (1 + \sum_{j=1}^{p(N)} P_j(s) n^{-j/2}\Big) ds + \mathcal O (n^{-N}) \,.
\end{equation}
In the last step, we see that, modulo an exponentially small error, we can integrate over $(0,+\infty)$ in order to get
\begin{equation}
\widetilde \sigma_n(\beta)=\int_0^{+\infty} e^{\beta s -\frac{s^2}{2}} s^{1/2} \Big (1 + \sum_{j=1}^{p(N)} P_j(s) n^{-j/2}\Big) ds + \mathcal O (n^{-N})\,.
\end{equation}

\vspace{0.1cm}\noindent
Hence we get by integration the following  lemma : \\
\begin{lemma} 
 \clb For any $N \in \N^*$, there exist $p(N) \in \N^*$, a polynomial  $P_1(s)$ and $C^\infty$-functions $\check \sigma_{j}$  such that
 \begin{equation}
 \sigma_n(\beta)=\int_0^{+\infty} e^{\beta s -\frac{s^2}{2}} s^{1/2} ds + \Big(\int_0^{+\infty} e^{\beta s -\frac{s^2}{2}} s^{1/2} P_1(s)  ds\Big)\, n^{-1/2} + \sum_{j=2}^{p(N)} \check \sigma_{j} (\beta) n^{-j/2} +  \mathcal O (n^{-N})\,.
 \end{equation}
Similarly,  for any $N\in \mathbb N^* $, there exist $p(N)\in \mathbb N^*$, a polynomial $Q_1(s)$    and $C^\infty$-functions $\check \tau _{j}$  such that
\begin{equation}
 \tau_n(\beta)=\int_0^{+\infty} e^{\beta s -\frac{s^2}{2}} s^{-1/2} ds + \Big(\int_0^{+\infty} e^{\beta s -\frac{s^2}{2}} s^{-1/2} Q_1(s)  ds\Big) \, n^{-1/2}  + \sum_{j=2}^{p(N)} \check \tau_{j} (\beta) n^{-j/2} +  \mathcal O (n^{-N})\,.
\end{equation}
\end{lemma}

\vspace{0.1cm}\noindent
Looking at the first term in the Taylor expansion, we see that
\begin{equation}
P_1(s)= s -\frac{s^3}{3}= Q_1(s)\,.
\end{equation}

\subsection{First localization of $\beta_n$.}

\vspace{0.1cm}

\begin{lemma}
For any $\eta >0$, there exists $n_0$ such that for $n\geq n_0$, 
\begin{equation}
\beta_n \in [\alpha -\eta,\alpha + \eta]\,.
\end{equation}
\end{lemma}

\begin{proof}
For fixed $\beta$, we have
\begin{equation}
\lim_{n\rightarrow +\infty}  \sigma_n(\beta)= \int_0^{+\infty} e^{\beta s -\frac{s^2}{2}} s^{1/2} ds\,,
\end{equation}
and
\begin{equation}
\lim_{n\rightarrow +\infty}  \tau_n(\beta)= \int_0^{+\infty} e^{\beta s -\frac{s^2}{2}} s^{-1/2} ds\,,
\end{equation}
This implies, 
\begin{equation}
\Phi(\beta):=\lim_{n\rightarrow +\infty}  \theta_n(\beta) = \frac{\int_0^{+\infty} e^{\beta s -\frac{s^2}{2}} s^{1/2} ds}{\int_0^{+\infty} e^{\beta s -\frac{s^2}{2}} s^{-1/2} ds}\,,
\end{equation}
or equivalently using (\ref{integralrep}),
\begin{equation}
	\Phi(\beta) = \half \frac{D_{-3/2} (-\beta)}{D_{-1/2}(-\beta)} \,,
\end{equation}
(we remark that, for any  $\nu<0$, $D_{\nu} (z)$ has no real zeros since the integrand in (\ref{integralrep}) is always positive).

\vspace{0.1cm}\noindent
Now, using (\ref{recurrenceDnu1}) with $\nu =-\half$ and $z=-\beta$, we immediately get :
\begin{equation}\label{newPhi}
	\Phi (\beta) = \beta + \frac{D_\half ( -\beta)}{D_{-\half} (-\beta)}.
\end{equation}
{\clb Considering now $\Psi_n$ given in (\ref{eq:as3})}, we get :
\begin{equation}
\lim_{n\rightarrow +\infty} \Psi_n(\beta) = \beta - \Phi(\beta) =  - \frac{D_\half ( -\beta)}{D_{-\half} (-\beta)}\,.
\end{equation}
Since  $D_\half ( -\alpha)=0$, one has $D_\half '( -\alpha)\neq 0$ since $D_\nu(z)$ satisfies the second order ordinary differential equation (\ref{ODEDnu}). It is then clear that for $\eta >0$ small enough, and $n$ large enough we have
\begin{equation}
\Psi_n(\alpha -\eta) \Psi_n(\alpha +\eta) <0\,,
\end{equation}
hence $\Psi_n$ should have a zero in this interval, which is necessarily $z_n$ by uniqueness.
This achieves the proof of the lemma.
\end{proof}

\vspace{0.1cm}\noindent
In other words, we have shown that
\begin{equation}\label{twoterm}
\lim_{n\rightarrow +\infty} \beta_n =\alpha\,,
\end{equation}
and as a consequence a two-terms asymptotics for $z_n$.

\subsection{Two-terms expansion for $\theta_n(\beta)$.}

\vspace{0.1cm}
In the previous subsection, we have seen that
\begin{equation}
\theta_n(\beta)= \Phi(\beta) + \mathcal O(n^{-1/2})\,.
\end{equation}
Our aim is now to compute the second term in the expansion. For this, we come back to the two-terms expansions of $\sigma_n$ and $\theta_n$. So we have
\begin{equation}
\theta_n(\beta) = \frac{A +  B n^{-\half} +\mathcal O(n^{-1})}{C+ D n^{-\half} + \mathcal O (n^{-1})}= \frac{A}{C} +  \frac{BC-AD}{C^2}\ n^{-\half} + \mathcal O (n^{-1})\,,
\end{equation}
with
\begin{eqnarray} \label{ABCD}
A &=& \int_0^{+\infty} e^{\beta s -\frac{s^2}{2}} s^{\half} \, ds\ \ ,\ 
B = \int_0^{+\infty} e^{\beta s -\frac{s^2}{2}} (s-\frac{s^3}{3}) \ s^{\half} \, ds\,, \\ 
C &=& \int_0^{+\infty} e^{\beta s -\frac{s^2}{2}} s^{-\half} \, ds\, \ ,\ 
D = \int_0^{+\infty} e^{\beta s -\frac{s^2}{2}}  (s- \frac{s^3}{3} )\ s^{-\half}\, ds\,,
\end{eqnarray}
where $A,B,C, D$ are functions of the variable $\beta$. For later use, we introduce
\begin{equation}\label{Delta}
\Delta(\beta) =  \frac{BC-AD}{C^2} = \left( \frac{D}{C} \right)'\,,
\end{equation}
since $B=D'$ and $A=C'$. Here the notation $C'$ and $D'$ stands for the derivative with respect to $\beta$.

\subsection{A complete asymptotic expansion.}

\vspace{0.1cm}
We come back to the expansions of $\sigma_n$ and $\tau_n$ and assuming that $\beta \in [\alpha-\eta, \alpha + \eta]$ all the control of the remainders are uniform for $\beta$ in this interval.

\vspace{0.1cm}\noindent
We can also obtain similar expansions for $\sigma'_n$ and $\tau'_n$, hence for $\theta'_n$ and {\clb $\Psi'_n$}. We can also get similar expansions for higher order {\clb derivatives}. We are actually in rather standard situation related to the implicit function theorem.

\begin{lemma} There exists $\eta_0>0$ and $n_0$ such that, for $n\geq n_0$, $\beta \in [\alpha-\eta_0, \alpha + \eta_0]$,
\begin{equation}
\Psi'_n(\beta) \geq \frac{1}{4}.
\end{equation}
\end{lemma}

\begin{proof}
First, using (\ref{newPhi}) and the relation $D_\half (-\alpha) =0$, we see that :
\begin{equation}
	\Phi'(\alpha) = 1- \frac{ D_\half' (-\alpha)}{D_{-\half}(-\alpha)}\,.
\end{equation}
Thanks to (\ref{recurrenceDnu2}) with $\nu = \half$ and $z =-\alpha$, we get immediately
\begin{equation}
	\Phi'(\alpha) = \half \,.
\end{equation}	
As one has 
 \begin{equation}
\lim_{n\rightarrow +\infty} \Psi_n'(\alpha) = 1 - \Phi'(\alpha) = \half \,.
\end{equation}
this concludes the proof.
\end{proof}

\subsubsection{Construction of the approximate solution.} 

We introduce in the sense of formal series,
 \begin{equation}
 \widehat \psi (\beta,\epsilon) = {\clb  \Psi_n(\beta)}
\end{equation}
 with $\epsilon=n^{-1/2}$.
We can now look at the construction of approximate solutions. We start from $\beta_n^{(0)} =\alpha$, which satisfies
 \begin{equation}
\hat \psi (\alpha,\epsilon) = \sum_{j\geq 1} r_j^{(0)} \epsilon^j\,.
\end{equation}
To get the second term, we look for a solution in the form :
\begin{equation}
\beta_n^{(1)} = \alpha + n^{-1/2} \alpha_1\,.
\end{equation}
 Using Taylor  at order $1$, we write :
 \begin{equation}
\hat \psi (\alpha + \epsilon \alpha_1,\epsilon)=  (r_1^{(0)}  + \hat \psi'(\alpha,0)\alpha_1)\epsilon + \mathcal O (\epsilon)\,.
\end{equation}
Choosing 
 \begin{equation}
\alpha_1= - r_1^{(0)} /  \hat \psi'(\alpha,0)\,,
\end{equation}
we get a solution modulo $\mathcal O (\epsilon^2)$. We can then iterate.

\subsubsection{Comparison between $\beta_n$ and its approximation.}

To simplify, we treat the case of the {\clb error} modulo $\mathcal O (\epsilon^2)$. We have :
 \begin{equation}
{\clb  \Psi_n(\beta_n)=0}\,,\, \mbox{ and } \psi_n (\beta_n^{(1)} )= \mathcal O (\epsilon^2)\,.
\end{equation}
 So we get, for some $\gamma_n \in [\beta_n,\beta_n^{1}]$, 
 {\clb 
 \begin{equation}
 \Psi_n(\beta_n) - \Psi_n (\beta_n^{(1)}) = (\beta_n-\beta_n^{1}) \Psi'_n (\gamma_n) = \mathcal O(\epsilon^2)
\end{equation}
 Using the lower bound of $\Psi'_n$ in $[\alpha-\eta,\alpha+\eta]$, we get :
}
 \begin{equation}
 \beta_n-\beta_n^{(1)} = \mathcal O (n^{-1})\,.
 \end{equation}
 The general case can be treated by recursion. So, using the results of the previous subsections, we get :

 \begin{equation}
 \beta_n^{(1)} = \alpha + \frac{\Delta(\alpha) +\alpha^2}{1-\Phi'(\alpha)}\ n^{-1/2} + \mathcal O (n^{-1})\,.
 \end{equation}

\vspace{0.5cm}\noindent
Notice that  $\Delta(\alpha)$ {\clb given in (\ref{Delta})} is simply related to $\alpha$. This is the object of the following lemma :

\vspace{0.1cm}
\begin{lemma}
	One has :
	\begin{equation}
		\Delta(\alpha) = \frac{1-10 \alpha^2}{12}.
	\end{equation}
\end{lemma}

\begin{proof}  First, {\clb using  (\ref{integralrep}) and (\ref{ABCD})}, we see that :
\begin{equation}
		C(\beta) = \sqrt{\pi} \  D_{-\half}(-\beta) \  e^{\frac{\beta^2}{4}}.
\end{equation}
So,  it follows from (\ref{ODEDnu}) that $C(\beta)$ satisfies the following ODE :
	\begin{equation}\label{ODEC}
		C''-\beta \ C' - \half \ C =0.
	\end{equation}
Moreover at $\beta=\alpha$, one has :
	\begin{equation}
		\frac{C'(\alpha)}{C(\alpha)} = - \frac{D_{-\half}'(-\alpha)}{D_{-\half} (-\alpha)} + \frac{\alpha}{2}.
	\end{equation}
Using (\ref{recurrenceDnu})	 with $\nu = -\half$ and $z= -\alpha$, we immediately get since $D_{\half}(-\alpha)=0$,
	\begin{equation}\label{logderiv}
		\frac{C'(\alpha)}{C(\alpha)} = \alpha\,.
	\end{equation}	
	
\vspace{0.2cm}
\noindent
Secondly, let us investigate $D(\beta)$, {\clb (in the following, we denote $C:=C(\beta), \ D:=D(\beta)$ to simplify the exposition)}. Clearly, one has :
	\begin{eqnarray*}
		D &=& C' - \frac{1}{3} \left( \int_0^{+\infty} e^{\beta s - \frac{s^2}{2}} \ s^2 \  s^{-\half} \ ds \right)' \\
		  &=&  C' - \frac{1}{3} \left( \int_0^{+\infty} e^{\beta s - \frac{s^2}{2}} \ (s^2 - \half) \  s^{-\half} \ ds  + \half \ C \right)' , \\
		  &=& C' - \frac{1}{3} \ \left( (\beta C'  + \ \half \ C) + \half \ C \right)' \,,
	\end{eqnarray*}
where in the last step we have used (\ref{ODEC}).	Using again (\ref{ODEC}), we get immediately :
\begin{equation}
	\frac{D}{C} = \frac{ (\half-\frac{\beta^2}{3}) C'}{C} - \frac{\beta}{6}\,.
\end{equation}
Repeating the argument, a straightforward calculus gives :
\begin{equation}
\left( \frac{D}{C} \right)' =  - \frac{C'}{C} \ \left( \frac{\beta}{6} + \frac{\beta^3}{3} + (\half - \frac{\beta^2}{3}) \frac{C'}{C} 	\right)	 + \frac{1-2\beta^2}{12}\,.
\end{equation}	
Finally, at $\beta=\alpha$ and using (\ref{logderiv}), we get :
\begin{equation}
	\Delta(\alpha) = \left( \frac{D}{C} \right)' (\alpha) = \frac{1-10 \alpha^2}{12}.
\end{equation}	
	
\end{proof}

\subsection{Proof of Proposition \ref{completeasympt}.}

Since $	\Phi'(\alpha) = \half $, it follows that 
\begin{equation}
\beta_n \sim \alpha + \frac{2\alpha^2 +1}{6} \ n^{-\half} + \sum_{j\geq 1} \alpha_j \ n^{-\frac {j+1}{2}}\,,
\end{equation}
which concludes the proof recalling that $\beta_n =\frac{z_n -n -\half}{\sqrt n}$.  \hfill $\Box$

\subsection{Applications.}
\vspace{0.2cm}\noindent
We first have the proposition :
\begin{prop}
\begin{equation}
\lim_{z\rightarrow +\infty} z^{-1/2} \, \lambda^{DN} (z) = \alpha\,.
\end{equation}
\end{prop}
\begin{proof}
Since $\lambda^{DN} (z)=\lambda_n (z)$ on $[z_{n-1}, z_n]$, we have :
\begin{equation}
\lambda_{n-1}(z_{n-1})\leq \lambda^{DN} (z) \leq \lambda_n(z_n)\,.
\end{equation}
Using the (F) formula at the points $z_{n-1}$ and $z_n$, we obtain :
\begin{equation}
z_{n-1} -n \leq \lambda^{DN} (z) \leq z_n -n -1\,,
\end{equation}
and we conclude using the two-terms asymptotics (\ref{twoterm}).
\end{proof}

\vspace{0.2cm}\noindent
Now, using the four-terms asymptotics (\ref{completeasympt}), one gets :

\begin{coro}\label{ecart}
We have
\begin{equation}
z_n- z_{n-1} = 1 + \frac{\alpha}{2} n^{-1/2} + \mathcal O (n^{-1})\,. 
\end{equation}
\end{coro}

\vspace{0.1cm}\noindent
Finally, using Formula (F), we get :
\begin{coro}\label{lambdazn}
\begin{equation}
\lambda_n(z_n)  =\alpha n^{1/2} + \frac{\alpha^2 -1}{3}  + \mathcal O (n^{-1/2})
\end{equation}
\end{coro}

\vspace{0.1cm}\noindent
Now, we can  get the asymptotic expansion for the ground state $\lambda^{DN} (z)$ as $z \to +\infty$. One has :

\vspace{0.1cm}\noindent
\begin{coro}
\begin{equation}
\lambda^{DN} (z)= \alpha z^{1/2} - \frac{\alpha^2 +2}{6} + \mathcal O (z^{-1/2})\,.
\end{equation}
\end{coro}
\begin{proof}
First, using  (\ref{eq:twoterms}), we easily see that :
\begin{equation}
\sqrt{n} =   \sqrt{z_n} - \frac{\alpha}{2}  + \mathcal O (n^{-1/2})\,.
\end{equation}
Thus, using Corollary \ref{lambdazn}, we obtain :
\begin{equation}\label{intermediaire}
	\lambda_n (z_n) = \alpha \sqrt{z_n} -  \frac{\alpha^2 +2}{6} + \mathcal O (n^{-1/2})\,.
\end{equation}
Now, for $z \in [z_{n-1},z_n]$, we recall that :
\begin{equation}
	\lambda_{n-1}(z_{n-1}) \leq  \lambda^{DN} (z) \leq \lambda_n (z_n)\,.
\end{equation}
Then, using (\ref{intermediaire}) and Corollary \ref{ecart}, a straightforward calculation gives the result.
\end{proof}

\vspace{0.1cm}\noindent
\begin{rem}\label{rayonR}
It is interesting to mention that  $-\frac{\alpha^2 + 2}{6} \approx - 0.430858$.  This corresponds to the curvature effect observed for the Neumann problem. To understand better, one should consider the disk of radius $R$ and verify that the first term is independent
 of $R$ and appears in the second term.  We have indeed :
 \begin{equation}\label{eq:scaling}
 \lambda^{\rm DN}(z, B_R)=\frac 1R \ \lambda^{\rm DN} (R^2 z, B_{R=1})\,.
 \end{equation}
 
 \end{rem}

\begin{rem}
We have in this way given a complete answer to the questions considered in Example 2.8 in \cite{CGHP}. A natural question is to prove this result for general domains (see next section).
\end{rem}

\section{Variational characterization and applications.}

If $\Omega$ is a bounded set in  $\mathbb R^d$, we  start from the variational characterization of the lowest  eigenvalue of the Dirichlet-to-Neumann operator  (see \cite{CPS}), for a given magnetic potential $A$.
\begin{equation}
\lambda^{\rm DN} (A, \Omega) = \inf_{f\in C^\infty(\bar \Omega)} \frac{ \int_\Omega |\nabla_A f|^2 \,dvol_\Omega}{\int_{\partial \Omega} |f|^2\, dvol_{\partial \Omega}}\,.
\end{equation}
This formula has various consequences which we describe in the next subsections.

\subsection{Conjectures}
If we have in mind that the ground state energy for the Neumann problem is given by
\begin{equation}
\lambda^{Ne} (A, \Omega) = \inf_{f\in C^\infty(\bar \Omega)} \frac{ \int_\Omega |\nabla_A f|^2 \,dvol_\Omega}{\int_{\Omega} |f|^2 \, d{\rm vol}_{ \Omega}}\,.
\end{equation}
it seems natural to hope that the techniques used for getting the large $b$ expansion of $\lambda^{Ne} (b,\Omega)$ (see \cite{FH2,HeLe} and references therein) could be adapted to get the asymptotic of  $\lambda^{\rm DN} (b,\Omega)$. The proof goes indeed through the analysis of the De Gennes model 
 for the half plane and the disk model  in $2$-D. Hence we can hope to prove :

\begin{conjecture} 
 	Let $\Omega$ be a regular domain in $\mathbb R^2$, $A_0$ be a magnetic potential with constant magnetic field with norm $1$, then the ground state energy of the D-to-N map $\Lambda_{bA_0}$ satisfies
 \begin{equation}
 \lim_{b\rightarrow +\infty} b^{-1/2} \; \lambda^{\rm DN} (b A_0,\Omega) = \alpha.
 \end{equation}
\end{conjecture}

\vspace{0.1cm}\noindent
One can actually hope to obtain  two-term asymptotics where the second term takes account of the curvature related to our computation of the second term (see Remark \ref{rayonR}):

\begin{conjecture} 
 	Let $\Omega$ be a regular domain in $\mathbb R^2$, $A_0$ be a magnetic potential with constant magnetic field with norm $1$, then the ground state energy of the D-to-N map $\Lambda^{DN}_{bA_0}$ satisfies
 \begin{equation}
b^{-1/2} \; \lambda^{\rm DN} (b A_0,\Omega) = \alpha - \frac{\alpha^2 +2}{6} \max_{x\in \partial \Omega} \kappa_x \,  b^{-1/2} + o(b^{-1/2})\,,
 \end{equation}
 where $\kappa_x$ denotes the curvature at $x$.
\end{conjecture}

\subsection{Diamagnetism}
\vspace{0.1cm}\noindent
 Note also that in the same way one proves :
 \begin{equation}
 \lambda^{Ne} (A,\Omega) \geq \lambda^{Ne} (A=0,\Omega)\,,
 \end{equation}
 for the ground state energy of the Neumann problem in $\Omega$,
 one obtains the diamagnetic inequality relative to the D-to-N operator:
 \begin{equation}
  \lambda^{\rm DN} (A,\Omega)\geq  \lambda^{\rm DN} (0,\Omega)\,.
 \end{equation}

\vspace{0.1cm}\noindent
The case of equality is considered in \cite{CGHP}. In the same way of the proof of \cite{He}, (see also \cite{Hebook} and  also independently  Shigekawa \cite{Sh}), one can recover Theorem 2.1 in \cite{CGHP}.
 Notice  that the proof of \cite{He} works also in presence of some {\clb electric} potential $V$.
 \begin{conjecture}
 For any bounded set $\Omega \subset \R^d$ and constant magnetic field  we have strong diamagnetism.
 \end{conjecture}

\vspace{0.1cm}\noindent
Although proved  for  the Neumann problem for sufficienly large magnetic field (see \cite{FH2}), this remains a conjecture (see \cite{FH3}) for all $b$.

 \subsection{Comparison between $\Theta_0$ and  $\alpha$.}
We use the variational characterization of $\alpha$. We have
\begin{equation}
\alpha = \inf_{f\in C_0^{\infty}(\overline{\mathbb R_+}),\xi \in \mathbb R} \frac{ \int_0^{+\infty} \big(|f'(t)|^2+ |(2t-\xi)f(t)|^2\big) dt}{f(0)^2}\,.
\end{equation}
 Taking $\xi=\sqrt{2}\xi_0$ and\footnote{To be rigourous, we should consider a sequence $f_n = f \chi_n$ where $\chi_n$ has support in $[0,n]$}   $f(t)=u_0( \sqrt{2}t)$, where $u_0$ is the ground state of $\mathfrak h(\xi_0)$, we get using (\ref{eigenf})- (\ref{eigenf0}),
\begin{equation}
 \alpha \leq \sqrt{2}\ \frac{\Theta_0}{u_0 (0)^2}\sim 1.0946  \,.
\end{equation}
We have indeed (see \cite{HeLe})
\begin{equation}
	u_0(0)^2  \sim  0.7622\,.
\end{equation}
This comparison is not very sharp but the spirit of the proof could be interesting to get rough estimates for other domains.

 \subsection{Comparison between $\lambda_{-n}(b)$ and $\lambda_n(b)$.}
We use the  variational characterization for $\lambda_n(b)$ which reads
\begin{equation}
\lambda_n(b)= \inf_{f\in C_0^\infty((0,1])} \frac{ \int_0^1 \big(f'(r)^2 + (br-\frac{n}{r})^2 f(r)^2\big)  r dr}{f(1)^2}\,.
\end{equation}
It is then immediate that for $b>0$ and $n\geq 0$ we have \eqref{eq:n-n}.

\vspace{0.5cm}

\noindent \footnotesize{
	
	\noindent Laboratoire de Math\'ematiques Jean Leray, UMR CNRS 6629. Nantes Universit\'e  F-44000 Nantes  \\
	\emph{Email adress}: Bernard.Helffer@univ-nantes.fr \\

	\noindent Laboratoire de Math\'ematiques Jean Leray, UMR CNRS 6629. Nantes Universit\'e  F-44000 Nantes \\
	\emph{Email adress}: francois.nicoleau@univ-nantes.fr \\

\end{document}